\documentclass[10pt,a4paper]{amsart}
\usepackage{amsthm}
\usepackage{amsmath}
\usepackage{mathtools}
\usepackage{amscd}
\usepackage{braket}
\usepackage[T1]{fontenc}
\usepackage[mathscr]{eucal}
\usepackage{MnSymbol}
\usepackage{parskip}
\numberwithin{equation}{section}
\usepackage[margin=2.9cm]{geometry}
\usepackage{tikz-cd}

\theoremstyle{definition}
\newtheorem{Th}{Theorem}[section]
\newtheorem{Lemma}[Th]{Lemma}

\newtheorem{Prop}[Th]{Proposition}

\newtheorem{Rem}[Th]{Remark}

\newcommand{\Heis}{S}
\newcommand{\pHeis}{S_{1}}

\newcommand{\twomodforms}{\mathcal{M}_{2}}
\newcommand{\pmodforms}{\mathcal{M}_{p}}
\newcommand{\poverconv}[1]{\pmodforms^{\dagger} (#1)}
\newcommand{\twooverconv}[1]{\twomodforms^{\dagger} (#1)}
\newcommand{\haupt}{\lambda}
\newcommand{\upee}[1]{\left. #1 \right\rvert \hspace{1pt} U_{p}}
\newcommand{\utwo}[1]{\left. #1 \right\rvert \hspace{1pt} U_{2}}
\newcommand{\uthree}[1]{\left. #1 \right\rvert \hspace{1pt} U_{3}}
\DeclareMathOperator{\SL}{SL}

\begin{document}

\title[Image of the 2-adic Character Map via MLDEs]{An Improved Lower Bound on the Image of the 2-adic Character Map for the Heisenberg Algebra via Modular Linear Differential Equations}

\author{Daniel Barake and Cameron Franc}

\begin{abstract}
We describe families of MLDEs whose solutions are modular forms of level one that converge, $2$-adically, to a Hauptmodul on $\Gamma_0(2)$ by using a theorem of Serre. Then, we apply this to show that the image of the character map on the $2$-adic Heisenberg VOA $S_{1}$ contains the space of $2$-adic overconvergent modular forms $M_{2}^{\dagger}(1/2)$ of weight zero.
\end{abstract}

\maketitle

\section{Introduction}

The study of $p$-adic vertex operator algebras was initiated by Franc-Mason \cite{FM23} in part as a means to further incorporate the $p$-adic study of modular forms into VOA theory. In particular, the normalized character map (i.e. 1-point correlation function or graded trace) on these new structures was shown to have image inside Serre's space of $p$-adic modular forms \cite{pSerre}. In the algebraic rather than the $p$-adic case, it is known that the normalized character map $F_{\Heis}$ on the Heisenberg VOA $\Heis$ gives a surjection of graded linear spaces $\Heis \to M'$ where $M'$ denotes the space of quasi-modular forms --- see \cite{MT} for a proof. Whether or not the $p$-adic character map surjects onto the ring of $p$-adic modular forms is currently an open question. 

For the $p$-adic completion of $\Heis$, which we denote as $\pHeis$, it was shown in \cite{BF24} that the image of the character map contains families of $p$-adic modular forms, including every $p$-adic Eisenstein series $G_{k}^{\star}(q)$. In fact, a similar result was proven for a $p$-adic completion of a lattice vertex operator algebra, though in this manuscript we focus only on $\pHeis$. Then in \cite{FM23a}, the authors used the description of the space of $2$-adic modular forms as the Tate ring $\twomodforms = \mathbb{Q}_{2}\langle j^{-1} \rangle$ in the variable $j^{-1}$, where $j$ is the classical $j$-invariant, to prove that the space $\twooverconv{7/4}$ of overconvergent modular forms of weight zero lies in the image of the character map on $\pHeis$. See Section \ref{Subsection: padicforms} for a brief introduction to $p$-adic and overconvergent modular forms, and in particular for a definition of the space $M^\dagger_2(r)$. In Remark 6.8 of \cite{FM23a}, it was proposed that instead of working with the description $\twomodforms = \mathbb{Q}_{2}\langle j^{-1} \rangle$, the alternative description $\twomodforms = \mathbb{Q}_{2}\langle \haupt \rangle$, where $\haupt$ is a genus-zero Hauptmodul on $\Gamma_{0}(2)$, might yield better $p$-adic properties. This would be in analogy with similar computations in number theory, cf. \cite{BC05,BKK05}. Indeed, our first main theorem is the following:
\begin{Th} \label{Th: TheoremImage}
    The space $\twooverconv{1/2}$ of $2$-adic overconvergent modular forms of level one and weight zero lies in the image of the character map on $\pHeis$.
\end{Th}

Since there are containments $\twooverconv{r} \subseteq \twooverconv{s} \subseteq \mathcal{M}_2$ when $s \leq r$, Theorem \ref{Th: TheoremImage} represents an improvement over the analogous result of \cite{FM23a}.

Though our focus will be when $p=2$, the hauptmodul $\haupt \in M_{0} \left( \Gamma_{0}(p) \right)$ over any prime $p \in \lbrace 2,3,5,7,13 \rbrace$ is given by the expression
\begin{align*}
    \haupt = \left( \frac{\eta (q^{p})}{\eta (q)} \right)^{\frac{24}{p-1}}.
\end{align*}
This is a meromorphic function on the compact modular curve $X_{0}(p)$ possessing a simple zero at the cusp and a pole at zero. Here, $\eta (q)$ is the Dedekind-eta function of weight $1/2$ satisfying $\eta (q)^{24} = \Delta (q)$ where
\begin{align*}
    \Delta (q) = \frac{E_{4}(q)^{3} - E_{6}(q)^{2}}{1728}
\end{align*}
is the modular discriminant of weight 12, and
\begin{align*}
    E_{k}(q) = 1 - \frac{2k}{B_{k}} \sum_{n \geq 1} \sigma_{k-1}(n) q^{n}
\end{align*}
denotes the normalized Eisenstein series of level one of weight $k$ on $\SL_{2}(\mathbb{Z})$. 

Much of the proof of Theorem \ref{Th: TheoremImage} is dedicated to showing that the powers $\haupt^{n}$ can be realized as the characters of specific states in $\pHeis$. Hence, if we ignore the convergence conditions underlying the definition of the overconvergent spaces, then Theorem \ref{Th: TheoremImage} immediately implies the following simpler statement:
\begin{Th} \label{Th: TheoremImage2}
    The space $\mathbb{Q}_{2}[ \haupt ]$ lies in the image of the character map on $\pHeis$ where $\haupt = \Delta (q^{2}) / \Delta (q)$.
\end{Th}
Now in all of \cite{FM22,BF24,FM23a}, the character on $\pHeis$ is defined via $p$-adic completion, and it is in particular the linear map $\widehat{F}_{\Heis}$ making the following diagram commute:
\[
\begin{tikzcd}[row sep = 4em, column sep = 6em]
    \pHeis \arrow[r,"\widehat{F}_{\Heis}"] & \twomodforms \\ \Heis \arrow[u,hook] \arrow[r,"F_{\Heis}"] & \mathbb{Q}_{2}[E_{2},E_{4},E_{6}] \arrow[u,hook]
\end{tikzcd}
\]
Hence, we summarize the process of showing that specific $p$-adic modular forms arise as characters here using this diagram: First we find a sequence of modular forms of level one $(\haupt_{n,m}) \in M$ whose $p$-adic limit is $\haupt^{n}$ for each $n \geq 1$. Then we find a sequence $(v_{n,m}) \in \Heis$ of states satisfying that $F_{\Heis}(v_{n,m},q) = \haupt_{n,m}$. There is some choice for these pre-images, but we show that a rather natural choice produces the sequence $(v_{n,m})$ that converges $2$-adically in $S_1$ with respect to the sup-norm. Finally, we assess the $2$-adic properties of these $v_{n,m}$ along with the description of $\twooverconv{r}$ to show that indeed $\twooverconv{1/2}$ is contained in the image of the $p$-adic character map.

We say a bit more about the first step, realizing $\lambda^n$ as a sequence of modular functions of level one: In his 1973 paper \cite{pSerre}, Serre proved that modular forms $f$ on $\Gamma_{0}(p)$ are $p$-adic by constructing an explicit sequence $(f_{m})$ of forms of level one converging $p$-adically to $f$; these are obtained by ``tracing'' over representatives of $\Gamma_{0}(p) / \text{SL}_{2}(\mathbb{Z})$. We apply this method directly to $\haupt^{n}$ in order to obtain our sequence $(\haupt_{n,m})$ with $\lim_{m \to \infty} \haupt_{n,m} = \haupt^{n}$. 

The next step is to pass to the Heisenberg $\Heis$ using eq. (\ref{MTeq}) of Mason-Tuite in \cite{MT}, that is, we find pre-images of the modular forms $\haupt_{n,m}$ discussed above. To use the Mason-Tuite formula, it is necessary to describe the modular forms $\haupt_{n,m}$ in terms of the Eisenstein basis of $M$. It is not immediately obvious how to obtain such expressions from Serre's trace procedure that underlies the definition of $\haupt_{n,m}$. We resolve this issue via the following striking fact which serves as our second main theorem:
\begin{Th} \label{Th: TheoremMLDE}
    When $p=2$, the Serre sequence $(\haupt_{n,m})$ satisfying that $\haupt^{n} = \lim_{m \to \infty} \haupt_{n,m}$ for all $n \geq 1$ satisfies a third-degree modular linear differential equation for all $m \geq 1$.
\end{Th}
The exact equation is of generalized hypergeometric $_{3} F_{2}$-type, and it is given in the statement of Proposition \ref{Prop: hmdiffs} below. Of course, it is known that all modular forms satisfy differential equations, but the degree of these equations generally increases with the weight of the form. What we find striking in this result is that, by construction, the weights of the forms $\haupt_{n,m}$ in general grow with $n$ and $m$, yet the degree of the corresponding MLDE is always three.

With Theorem \ref{Th: TheoremMLDE} in hand, and knowing that the corresponding differential equation is generalized hypergeometric, we solve the MLDE using methods found in \cite{FM16,KT01}, and this solution allows us to write our modular forms in the Eisenstein basis $\haupt_{n,m} \in \mathbb{Z}[E_{4},E_{6}]$ for all $n,m \geq 1$. With such explicit descriptions in hand, we can then use the formula of Mason-Tuite to lift $\haupt_{n,m}$ to the Heisenberg algebra. We then conclude the proof of Theorem \ref{Th: TheoremImage} by studying the $p$-adic properties of these explicit formulas for the pre-images.

When $p \in \lbrace 3,5,7,13 \rbrace$, we find analogous sequence terms $\haupt_{n,m} \in M$ which converge to their respective hauptmoduln on $\Gamma_{0}(p)$ respectively, however, we are currently unable to find an MLDE such as the one referred to in Theorem \ref{Th: TheoremMLDE} which is satisfied by these $\haupt_{n,m}$ for the odd genus zero primes. At the very least, computational evidence using SAGE suggests that there is no low-degree (monic or non-monic) MLDE which has $\haupt_{n,m}$ as a solution for these primes. Besides providing general formulas for the $\haupt_{n,m}$ in Subsection \ref{Subsection: Serre's Sequence} below that work for all genus zero primes, we leave the search for analogous MLDEs --- or an alternative method for studying the image of the $p$-adic character map --- in these cases to the interested reader.

The paper is divided as follows: In $\S$\ref{Section: Background}, we review the necessary basic facts about both modular forms and the Heisenberg vertex operator algebra, as well as their $p$-adic counterparts. Then in $\S$\ref{Section: MLDES} we establish Theorem \ref{Th: TheoremMLDE} and we find explicit expressions for each sequence term $\haupt_{n,m}$. Finally $\S$\ref{Section: VOAS} is dedicated to finding the sequence of pre-images $(v_{n,m}) \in \Heis$, establishing their $p$-adic properties, and proving Theorems \ref{Th: TheoremImage} and \ref{Th: TheoremImage2}.

The authors would also like to thank the referee for their helpful comments on a prior version of this manuscript.

\section{Background and Notation} \label{Section: Background}

\subsection{Notation} We assume the following notational choices and conventions throughout:
\begin{itemize}
    \item $\mathbb{N} = \lbrace 0,1,2, \ldots \rbrace$ is the set of non-negative integers.
    \item $M = \mathbb{C}[E_{4},E_{6}]$ is the graded algebra of modular forms of level one. 
    \item $M' = \mathbb{C}[E_{2},E_{4},E_{6}]$ is the graded algebra of quasi-modular forms of level one.
    \item $M(\Gamma) = \bigoplus_{k \geq 0} M_{k}(\Gamma)$ is the graded algebra of modular forms on the congruence subgroup $\Gamma$. 
    \item $\pmodforms$ is the space of $p$-adic modular forms of weight zero. 
    \item $\poverconv{r}$ is the space of $p$-adic $r$-overconvergent modular forms of weight zero.
    \item $B_{k}$ is the $k^{\text{th}}$ Bernoulli number. 
    \item $\tau$ is the complex variable in the Poincar\'e upper-half plane.
    \item $q = e^{2 \pi i \tau}$ is the nome.
    \item $G_{k}(q) = -\frac{B_{k}}{2k}E_{k}(q) = - \frac{B_{k}}{2k} + \sum_{n \geq 1} \sigma_{k-1}(n)q^{n}$ is the weight $k$ Eisenstein series. 
    \item $E_{k}(q) = - \frac{2k}{B_{k}}G_{k}(q)$ is the normalized weight $k$ Eisenstein series with constant coefficient 1.
    \item $\Delta (q) = \frac{1}{1728} ( E_{4}(q)^{3} - E_{6}(q)^{2})$ is the modular discriminant of weight $12$.
    \item $\eta (q) = q^{1/24} \prod_{i \geq 1} (1-q^{n})$ is the Dedekind $\eta$-function of weight $1/2$.
    \item $j = E_{4}(q)^{3} / \Delta (q)$ is the modular $j$-invariant.
    \item $E_{k}^{\star}(q) = E_{k}(q) - p^{k-1} E_{k}(q^{p})$ is the $p$-adic weight $k$ Eisenstein series.
    \item $\Heis$ is the (algebraic) Heisenberg VOA of central charge $c_{V} = 1$.
    \item $\pHeis$ is the $2$-adic Heisenberg VOA.
\end{itemize}

\subsection{Modular Differential Operators}
Recall that the Ramanujan-Serre derivative $\mathcal{D}_{r}: M_{r}(\Gamma) \to M_{r+2}(\Gamma)$ acts on modular forms $f \in M_{r}(\Gamma)$ by the rule
\begin{align}
    \mathcal{D}_{r} (f) = \theta (f) - \frac{r}{12} E_{2}(q) f = \theta(f) + 2r G_{2}(q)f \label{modularderdef}
\end{align}
where $\theta = q \frac{d}{dq}$. Define also the iterates
\begin{align*}
    \mathcal{D}_{r}^{n} = \mathcal{D}_{r+2(n-1)} \circ \mathcal{D}_{r+2(n-2)} \circ \cdots \circ \mathcal{D}_{r},
\end{align*}
keeping in mind that $\mathcal{D}_{r}^{2} \neq \mathcal{D}_{r} \circ \mathcal{D}_{r}$. The modular derivative satisfies the Leibniz rule in the following way: For $fg \in M_{r}(\Gamma)$, we have $\mathcal{D}_{r}(fg) = \mathcal{D}_{a}(f) \cdot g + f \cdot \mathcal{D}_{b}(g)$ where $r=a+b$. Hence we will often omit writing the subscript in the notation for the modular derivative when the weight is clear. The operator $\theta$ will also play an important role, and its interaction with Eisenstein series is given by the well-known Ramanujan identities
\begin{align}
    \theta (E_{2}(q)) &= \frac{1}{12} \left( E_{2}(q)^{2} - E_{4}(q) \right) \label{ram1} \\
    \theta (E_{4}(q)) &= \frac{1}{3} \left( E_{2}(q)E_{4}(q) - E_{6}(q) \right) \label{ram2} \\
    \theta (E_{6}(q)) &= \frac{1}{2} \left( E_{2}(q)E_{6}(q) - E_{4}(q)^{2} \right). \label{ram3}
\end{align} 
A monic degree-$t$ modular linear differential equation (MLDE) over $M'$ is then an expression of the form
\begin{align*}
    c_{t} \mathcal{D}^{t}(f(q)) + c_{t-1} \mathcal{D}^{t-1}(f(q)) g_{t-1}(q) + \cdots + c_{0} g_{0}(q) = 0
\end{align*}
where $g_{i}(q) \in M'_{2t-2i}$ for some coefficients $c_{i} \in \mathbb{C}$, and where $f(q)$ is called a solution of the MLDE. As a modular linear differential operator or ``MLDO'', it takes modular forms of weight $n$ to those of weight $n+2t$. \par

\subsection{\textit{p}-adic modular forms} \label{Subsection: padicforms}

Our description of $p$-adic modular forms follows the classical one of Serre in \cite{pSerre}. That is, a $p$-adic modular form is a formal power series
\begin{align}
    f = \sum_{i=0}^{\infty} a_{i}q^{i} \in \mathbb{Q}_{p}[[q]] \label{padicformex}
\end{align}
where $f$ is the limit of a sequence of modular forms $f_{m} \in M$ having weight $k_{m}$. Note the limit here is taken over the sup-norm; for $g = \sum_{j=0}^{\infty} c_{j}q^{j} \in \mathbb{Q}[[q]]$, this is
\begin{align*}
    \nu_{p}(g) = \inf_{j} \nu_{p}(c_{j}).
\end{align*}
The space $\pmodforms$ is then defined as the completion of $\mathbb{Q}_{p}[E_{4},E_{6}]$ with respect to the sup-norm above. We note also that the weight of $f$ in (\ref{padicformex}) is an element of the space
\begin{align*}
    X = \varprojlim_{m} \mathbb{Z}/p^{m}(p-1)\mathbb{Z} \cong \mathbb{Z}_{p} \times \mathbb{Z} / (p-1)\mathbb{Z}
\end{align*}
where $X \cong \mathbb{Z}_{2}$ when $p=2$. Important examples of $p$-adic modular forms are the $p$-adic Eisenstein series, which we will be repeatedly expressing in the form
\begin{align*}
    E_{k}^{\star}(q) = E_{k}(q) - p^{k-1}E_{k}(q^{p}), \hspace{10pt} k \geq 2.
\end{align*}
Again our focus will be when $p=2$, and we showcase some important identities in this case. We have $M(\Gamma_{0}(2)) = \mathbb{C}[E_{2}^{\star}, E_{4}]$ (see pg. 3 of \cite{Gott19}), so the following are easy to check:
\begin{align}
    E_{4}^{\star}(q) &= -10 E_{2}^{\star}(q)^{2} + 3 E_{4}(q) \label{e4starexp} \\
    E_{6}^{\star}(q) &= 40 E_{2}^{\star}(q)^{3} - 9 E_{2}^{\star}(q) E_{4}(q) \label{e6starexp}
\end{align}
Along with the Ramanujan identities (\ref{ram1})-(\ref{ram3}) we may compute their modular derivatives:
\begin{align}
    \mathcal{D} \left( E_{2}^{\star}(q) \right) &= \frac{1}{3} E_{2}^{\star}(q)^{2} - \frac{1}{6} E_{4}(q) \label{e2starder} \\
    \mathcal{D} \left( E_{4}^{\star}(q) \right) &= -\frac{8}{3} E_{2}^{\star}(q)^{3} + \frac{1}{3} E_{2}^{\star}(q) E_{4}(q). \label{e4starder}
\end{align}
These will be used later in Subsection \ref{Subsection: Sequences of MLDES} when deriving the MLDE in Theorem \ref{Th: TheoremMLDE}.

Theorem \ref{Th: TheoremImage} is concerned with the subspaces $\poverconv{r} \subset \pmodforms$ of $r$-overconvergent forms. Here we outline relevant facts from Section 3 of \cite{Vonk} as well as \cite{BC05}, in particular when $p=2$. The space of $2$-adic modular forms has the description 
\begin{align*}
    \twomodforms \cong \mathbb{Q}_{2}\langle j^{-1} \rangle = \Set{ a_{0} + a_{1}j^{-1} + a_{2}j^{-2} + \cdots | \vert a_{n} \vert_{2} \to 0}
\end{align*}
where $\vert \cdot \vert_{2}$ denotes the $2$-adic absolute value. For $r \geq 0$, the spaces $\twooverconv{r}$ are Banach spaces inside $\twomodforms$ which have additional growth conditions on the coefficients, hence the name ``overconvergent''. Alternatively, and motivated by number theoretical methods (cf. \cite{BC05,BKK05}), we have the description $\twomodforms \cong \mathbb{Q}_{2} \langle \haupt \rangle$ where $\haupt = \Delta (q^{2}) / \Delta (q)$, and the space of $r$-overconvergent forms becomes
\begin{align}
    \mathcal{M}_{2}^{\dagger}(r) = \mathbb{Q}_{2}\langle 2^{12r} \lambda \rangle = \Set{ a_{0} + a_{1}\haupt + a_{2}\haupt^{2} + \cdots | \vert a_{n} \vert_{2} \cdot 2^{12nr} \to 0 }.
\end{align}
The functions $j^{-1}$ and $\haupt$ are related via the expression
\begin{align*}
    j^{-1} = \frac{\haupt}{(1+2^{8}\haupt)^{3}},
\end{align*}
and we note that $\twooverconv{r} \subseteq \twooverconv{s}$ whenever $s \leq r$. In Subsection \ref{Subsection: Pre-images}, we construct a sequence $(v_{n,m}) \in \Heis$ and in Subsection \ref{Subsection: Convergence of Pre-Images}, we prove that its $2$-adic limit $v_{n}$, which satisfies $\widehat{F}_{\Heis}(v_{n},q) = \haupt^{n}$, lies in $S_{1}$ for all $n \geq 1$. The proof of Theorem \ref{Th: TheoremImage} is found in Subsection \ref{Subsection: OverconvergentImage} where we show that for all $n,m \geq 1$, $\nu_{2}(v_{n,m}) \geq -6n$. From the linearity of $\widehat{F}_{\Heis}$, this will show that the space
\begin{align*}
    \twooverconv{1/2} = \mathbb{Q}_{2} \langle 2^{6} \lambda \rangle = \Set{a_{0} + a_{1} \widehat{F}_{\Heis}(v_{1},q) + a_{2}\widehat{F}_{\Heis}(v_{2},q) + \cdots | \vert a_{n} \vert_{2} \cdot 2^{6n} \to 0}
\end{align*}
lies in the image of $\widehat{F}_{\Heis}$.

\subsection{The Heisenberg VOA}

For the purposes of this manuscript, we only require a few basic facts about the Heisenberg VOA $\Heis$, though we refer the reader to \cite{LL} for an excellent overview of its construction. We also assume $\Heis$ has central charge $c_{V} = 1$ throughout. The VOA $\Heis$ can be interpreted as the symmetric algebra in the variables $h(-n)$ where $n \geq 1$, along with a distinguished element $\textbf{1}$ acting as the identity. The terms $h(n)$ for $n \geq 1$ act as scalar partial derivatives $n \partial_{h(-n)}$, and $h(0)$ acts as zero. Each monomial vector in $\Heis$, or ``state'', is assigned a weight given by
\begin{align}
    \text{wt}(h(-n_{1}) h(-n_{2}) \cdots h(-n_{r})\textbf{1}) = \sum_{i=1}^{r} n_{i} \label{stateweight}
\end{align}
so that $\Heis$ possesses a $\mathbb{Z}_{\geq 0}$ grading. A key feature of $\Heis$ and in fact of any VOA is its ``square-bracket'' counterpart introduced in \cite{Zhu}. This is a VOA isomorphic to $\Heis$ however now with endomorphisms $h[n]$ where for $n \leq -1$ these act again by multiplication and as scalar partial derivatives for $n \geq 1$. Explicitly, these are given in terms of the round-bracket variables by expanding
\begin{align*}
    h[n] = \text{Res}_{z} Y(h(-1)\textbf{1},e^{z}-1)(e^{z})(z^{n})
\end{align*}
where $Y$ denote the vertex operator on $\Heis$. We not refer to this description for the remainder of the discussion, and in fact, we will only require the following expressions which can be easily computed from above:
\begin{align}
    &h[-2] = h(-2) + h(-1) - \frac{1}{240}h(2) + \frac{1}{240}h(3) - \cdots \label{squareh2exp} \\
    &h[-3] = h(-3) + \frac{3}{2}h(-2) + \frac{1}{2}h(-1) + \frac{1}{240}h(1) -\frac{1}{480}h(2) + \frac{1}{945}h(3) - \cdots \label{squareh3exp}
\end{align}
For the square-bracket Heisenberg VOA, we then have another grading analogous to (\ref{stateweight}) defined in the obvious way. We note, however that for a VOA $V$ we only have the equality
\begin{align*}
    \bigoplus_{n \leq N} V_{(n)} = \bigoplus_{n \leq N} V_{[n]} \hspace{15pt} N \in \mathbb{Z}
\end{align*}
where $V_{(n)}$ denotes the homogeneous degree-$n$ subspace of $V$, and $V_{[n]}$ is the same in the square-bracket grading. This follows from the expansion of the square-bracket Virasoro modes as
\begin{align*}
    L[0] = L(0) + \sum_{k \geq 1} \frac{(-1)^{k-1}}{k(k+1)}L(k)
\end{align*}
as well as the fact that the grading on a VOA coincides with its $L(0)$-eigenspace decomposition. The point is that states of $\Heis$ of square-bracket weight $k$ are modular forms of level one and weight $k$. To be more precise, let $F_{\Heis}$ denote the normalized character of $v \in V$:
\begin{align*}
    F_{\Heis}(v,q) = \eta (q) Z_{\Heis}(v,q) = q^{-1/24} \eta (q) \sum_{n \geq 0} \text{Tr}_{\Heis_{(n)}}o(v)q^{n}
\end{align*}
where $o(v)$ is the unique endomorphism associated to $v$ which preserves the round-bracket grading of $\Heis$. Note above that the trace is taken over the \textit{round}-bracket grading. We use the following re-formulation of equation (44) of \cite{MT} in order to compute $F_{\Heis}(v,q)$ for a state $v = h[-k_{1}]h[-k_{2}] \cdots h[-k_{r}]\textbf{1}$ for $k_{i} \geq 1$ not necessarily distinct:
\begin{align}
    F_{\Heis}(v,q) = \sum_{(\ldots \lbrace s,t \rbrace \ldots) \in \mathcal{P}(\Phi,2)} \prod_{\lbrace s,t \rbrace} \frac{2(-1)^{s+1}}{(s-1)!(t-1)!}G_{s+t}(q). \label{MTeq}
\end{align}
Here, $\mathcal{P}(\Phi,2)$ denotes all partitions of the set $\Phi = \lbrace k_{1}, k_{2}, \ldots , k_{r} \rbrace$ into parts $\lbrace s,t \rbrace$ of size two. Then for each such partition, the product is taken over each part. With this, one can show that $F_{\Heis}: \Heis \to M'$ is a surjection of graded linear spaces. We note however that $F_{\Heis}$ has large kernel, as for example, any state of odd weight is mapped to zero. 

A major technical aspect of the process of computing characters in $\pHeis$ is that the $p$-adic convergence of pre-images needs to be assessed in the round-bracket formalism and not in the square-bracket one, despite the fact that the description of the pre-images via eq. (\ref{MTeq}) is given in the latter. This technicality will be addressed in Subsections \ref{Subsection: Convergence of Pre-Images} and \ref{Subsection: Pre-images} by constructing an isomorphism of vector spaces between a certain subspace $\Sigma \subset \Heis$ and $M'$, then using the formulas (\ref{squareh2exp}) and (\ref{squareh3exp}) above to obtain a lower bound on the $2$-adic valuation of the pre-images $v_{n,m}$, rather than an explicit formula in the round-bracket formalism.

\subsection{The \textit{p}-adic Heisenberg VOA}

We briefly recall some notions on $p$-adic VOAs also stated in \cite{BF24}, however the inclined reader should consult \cite{FM23} for all of the necessary details. In general, given an algebraic VOA $U$ over $\mathbb{Z}_{p}$, one can define the $\mathbb{Z}_{p}$-module
\begin{align*}
    \widehat{U} = \varprojlim_{m} U / p^{m} U
\end{align*}
and set $\widehat{V} = \widehat{U} \otimes_{\mathbb{Z}_{p}} \mathbb{Q}_{p}$ which carries the structure of a $p$-adic VOA. Equivalently one can define $\widehat{V}$ in terms of a completion with respect to a sup-norm arising from an integral basis of $U$. Now when $V = \Heis$ which we view as algebraic over $\mathbb{Q}_{p}$, we define the space $S_{r}$ for $r > 0$ as the completion of $\Heis$ with respect to the sup-norm
\begin{align*}
    \left\lvert \sum_{I} a_{I}h^{I} \right\rvert_{r} = \sup_{I} \vert a_{I} \vert r^{I}
\end{align*}
for a finite multi-subset $I$ of $\mathbb{Z}_{<0}$ with $\vert I \vert = - \sum_{i \in I} i$ and where $h^{I} = \prod_{i \in I} h(i)$. Elements of $S_{r}$ are then of the form $\sum_{I} a_{I}h^{I}$ where
\begin{align*}
    \lim_{\vert I \vert \to \infty} \vert a_{I} \vert r^{\vert I \vert} = 0.
\end{align*}
The space $S_{1}$ is a $p$-adic VOA, and it will be our focus for the remainder of the manuscript. Note for different values of $r$, one has a filtration $S_{1} \supseteq S_{2} \supseteq S_{3} \supseteq \cdots $.

\section{Modular Differential Equations} \label{Section: MLDES}

\subsection{Serre's Sequence} \label{Subsection: Serre's Sequence}

Let $p \in \lbrace 2,3,5,7,13 \rbrace$. Following Th\'eor\`eme 10 of \cite{pSerre} and for any $n \geq 1$, we construct a sequence $\haupt_{n,m}$ of modular forms of level one satisfying $\haupt^{n} = \lim_{m \to \infty} \haupt_{n,m}$ by tracing over the representatives of $\Gamma_{0}(p) / \text{SL}_{2}(\mathbb{Z})$. We will see shortly that the $\haupt_{n,m}$ possess a pole of order $n$ at the zero cusp, hence we will be considering the sequence $\Delta (q)^{n} \haupt_{n,m} \in M_{12 \cdot 2^{m}+12n}$ satisfying that $\Delta (q)^{n} \haupt^{n} = \lim_{m \to \infty} \Delta (q)^{n} \haupt_{n,m}$. Only after realizing each as the solution of an MLDE in Proposition \ref{Prop: hmdiffs} and getting explicit expressions in Proposition \ref{Prop: hmnseisenstein} will we see that the $\haupt_{n,m}$ are modular of level one as expected, and we will divide by $\Delta (q)^{n}$. Let us set
\begin{align}
    \haupt_{n,m} = \haupt^{n} g^{3 \cdot p^{m}} + p^{1 - 6 \cdot p^{m}} \upee{ \left( \left. \haupt^{n}g^{3 \cdot p^{m}} \right\rvert_{12 \cdot p^{m}} \hspace{2pt} W \right) \hspace{2pt} }. \label{hmraweq}
\end{align}
Note that $\haupt_{n,m}$ has weight $12 \cdot p^{m}$, the form $g \in M_{4}(\Gamma_{0}(p))$ is given by
\begin{align*}
    g = E_{4}(q) - p^{4} \cdot E_{4}(q^{p}),
\end{align*}
and $f \mid_{r} W$ denotes application of the slash operator for a modular form $f$ of weight $r$ with
\begin{align*}
    W = \begin{pmatrix} 0 & -1 \\ p & 0 \end{pmatrix},
\end{align*}
and $U_{p}$ is the linear operator acting on $q$-series, given by the rule
\begin{align*}
    f \mid \hspace{1pt} U_{p} = \upee{ \left( \sum_{i} a_{i} q^{i} \right) } = \sum_{i} a_{pi} q^{i}.
\end{align*}

\begin{Lemma} \label{Lemma: Upoperator}
    Let $p$ be prime. For formal Laurent series $f(q), g(q) \in \mathbb{C}[[q]]$, we have the identities
    \begin{enumerate}
        \item $\upee{ \left( f(q) g(q^{p}) \right) } =  g (q) \upee{ \left( f (q) \right) }$
        \item $\upee{ \theta (f(q)) } = p \cdot \theta \left( \upee{ f(q) }  \right)$ where $\theta = q \frac{d}{dq}$
        \item For all $k \geq 2$ even, $\upee{ E_{k}^{\star}(q) }= E_{k}^{\star}(q)$.
    \end{enumerate}
\end{Lemma}
\begin{proof}
    Parts (1) and (2) are easily proven from the definitions. For part (3), see pg. 210 of \cite{pSerre}.
\end{proof}

It is easy to compute
\begin{align*}
    \left. \haupt^{n}g^{3 \cdot p^{m}} \right\rvert_{12 \cdot p^{m}} \hspace{2pt} W = p^{6 \cdot p^{m}} \left( p \tau \right)^{- 12 \cdot p^{m}} \left( \frac{\eta (-1/\tau)}{\eta (-1/p\tau)} \right)^{\frac{24n}{p-1}} \left( E_{4}(-1/p \tau) - p^{4} E_{4}(-1/\tau) \right)^{3 \cdot p^{m}},
\end{align*}
and with modular transformation laws
\begin{align*}
    \eta (-1/\tau) &= \sqrt{-i \tau} \cdot \eta ( \tau ) \\
    E_{4}(-1/ \tau) &= \tau^{4} E_{k}(\tau)
\end{align*}
and
\begin{align*}
    \eta (-1/ p\tau) &= \sqrt{-i p \tau} \cdot \eta ( p \tau ) \\
    E_{4} ( - 1 / p \tau ) &= (p \tau)^{4} E_{4}(p \tau)
\end{align*}
obtained by setting $\tau \mapsto p \tau$, we have
\begin{align*}
     \left. \haupt^{n}g^{3 \cdot p^{m}} \right\rvert_{12 \cdot p^{m}} \hspace{2pt} W = p^{6 \cdot p^{m} - \frac{12n}{p-1}} \cdot \haupt^{-n} \cdot \left( E_{4}(p \tau) - E_{4}(\tau) \right)^{3 \cdot p^{m}}
\end{align*}
so that we may re-write eq. (\ref{hmraweq}) as
\begin{align}
    \haupt_{n,m} = \haupt^{n}  \left( E_{4}(q) - p^{4} E_{4}(q^{p}) \right)^{3 \cdot p^{m}} + p^{1 - \frac{12n}{p-1}} \upee{ \left( \haupt^{-n} \left( E_{4}(q^{p}) - E_{4}(q) \right)^{3 \cdot p^{m}} \right) }.\label{hms}
\end{align}
and $\haupt^{n} = \lim_{m \to \infty} \haupt_{n,m}$ as desired.

\subsection{Sequences of MLDEs} \label{Subsection: Sequences of MLDES}

In this subsection, we formulate the explicit MLDEs needed to write the $\lambda_{n,m}$ in terms of an Eisenstein basis. In fact within Proposition \ref{Prop: hmdiffs}, we prove Theorem \ref{Th: TheoremMLDE}.

\begin{Lemma} \label{Lemma: hdiffeq}
Let $p=2$. For $n \geq 1$, $\Delta (q)^{n} \haupt^{n}$ satisfies the following third-degree MLDE:
\begin{align}
    \mathcal{D}^{3} \left( \Delta (q)^{n} \haupt^{n} \right) = \left( \left( n^{2} - n^{3} \right) E_{2}^{\star}(q)^{3} - \left( \frac{n^{2}}{2} + \frac{n}{18} \right) E_{2}^{\star}(q)E_{4}(q) \right) \Delta (q)^{n} \haupt^{n}.\label{hdiffeq}
\end{align}
\end{Lemma}
\begin{proof}
Note first that for all $i \geq 1$, $\mathcal{D}_{12n}^{i}(\Delta (q)^{n} \haupt^{n}) = \mathcal{D}_{0}^{i}(\haupt^{n})$ and so it suffices to show that $\haupt^{n}$ satisfies the MLDE in the claim with the modular derivative acting with weight zero. 

Since $\mathcal{D}( \Delta (q) ) = 0$ we have $\theta (\Delta (q)) = E_{2}(q) \Delta (q)$ and so
\begin{align*}
    \mathcal{D}(\Delta (q^{2})) = 2 E_{2}(q^{2}) \Delta (q^{2})  - E_{2}(q) \Delta ( q^{2}).
\end{align*}
Also since $\mathcal{D}(\haupt^{n}) = \mathcal{D}(\Delta ( q^{2} )^{n}) \cdot \Delta (q)^{-n}$, this implies
\begin{align}
    \mathcal{D}(\haupt^{n}) = -n \haupt^{n}E_{2}^{\star}(q). \label{modderh}
\end{align}
The identities (\ref{e4starexp})-(\ref{e4starder}) in turn allow us to compute
\begin{align*}
    \mathcal{D}^{2}\left( \haupt^{n} \right) = \left( n^{2} - \frac{n}{3} \right) \haupt^{n} E_{2}^{\star}(q)^{2} + \frac{n}{6} \haupt^{n} E_{4}(q),
\end{align*}
and from this we get
\begin{multline*}
    \mathcal{D}^{3} \left( \haupt^{n} \right) = \left( -n \left( n^{2} - \frac{n}{3} \right) + \frac{2}{3} \left( n^{2} - \frac{n}{3} \right) \right) \haupt^{n} E_{2}^{\star}(q)^{3} \\ + \left( - \frac{1}{3} \left( n^{2} - \frac{n}{3} \right) - \frac{n^{2}}{6} \right) \haupt^{n} E_{2}^{\star}(q) E_{4}(q) - \frac{n}{18} \haupt^{n} E_{6}(q),
\end{multline*}
so using the identity $E_{6}(q) = -4 E_{2}^{\star}(q)^{3} + 3 E_{2}^{\star}(q) E_{4}(q)$ this simplifies to
\begin{align*}
    \mathcal{D}^{3} \left( \haupt^{n} \right) = \left( \left( n^{2} - n \right) E_{2}^{\star}(q)^{2} + \left( \frac{n}{2} + \frac{1}{18} \right) E_{4}(q) \right) \mathcal{D}(\haupt^{n}),
\end{align*}
and the result follows from eq. (\ref{modderh}) above.
\end{proof}

We wish to show the $\Delta (q)^{n} \haupt_{n,m}$ satisfy a similar MLDE. Since there is no clear interaction between the $U_{2}$ operator and the modular derivative, we wish to simplify the second term of (\ref{hms}). To do this, let $\mathcal{E}_{4}(q)$ denote the normalized Eisenstein series of weight $4$ on $\Gamma_{0}(2)$ so that
\begin{align*}
    \mathcal{E}_{4}(q) = \frac{E_{4}(q) - E_{4}(q^{2})}{240} = q + 8q^{8} + 28q^{3} + 64 q^{4} + \cdots
\end{align*}
In fact, $\mathcal{E}_{4}(q)$ can be expressed as the eta-quotient
\begin{align*}
\mathcal{E}_{4}(q) = \frac{\eta (q^{2})^{16}}{\eta (q)^{8}}
\end{align*}
(cf. \cite{AZ19}, for example) which means
\begin{align*}
    \haupt^{-n} \mathcal{E}_{4}(q)^{3 \cdot 2^{m}} = \Delta (q^{2})^{2^{m+1}-n} \Delta (q)^{n-2^{m}}.
\end{align*}
Using Lemma \ref{Lemma: Upoperator} this gives
\begin{align*}
     \utwo{ \left( \haupt^{-n} \mathcal{E}_{4}(q)^{3 \cdot 2^{m}} \right) } = \Delta (q)^{2^{m+1}-n} \utwo{ \left( \Delta (q)^{n-2^{m}} \right) }.
\end{align*}
Putting everything together and using identities derived in Lemma \ref{Lemma: hdiffeq}, eq. (\ref{hms}) becomes
\begin{align}
    \haupt_{n,m} = \haupt^{n}  \left( 5E_{4}(q) - 20 E_{2}^{\star}(q)^{2} \right)^{3 \cdot 2^{m}} + 2^{1 - 12n} (240)^{3 \cdot 2^{m}} \Delta (q)^{2^{m+1}-n} \utwo{\left( \Delta (q)^{n-2^{m}} \right) }. \label{hmsbetter}
\end{align}
\begin{Rem}
    It may be possible to write (\ref{hmsbetter}) explicitly without the $U_{2}$-operator, using the recursive formula given in page 5 of \cite{BC05} or methods in \cite{D14}. Since it is interesting enough that the $\Delta (q)^{n} \haupt_{n,m}$ will satisfy a ``nice'' MLDE, we keep the expression for the $\haupt_{n,m}$ as it is written above.
\end{Rem}

\begin{Prop} \label{Prop: hmdiffs}
    Set $p=2$. For $n,m \geq 1$, the modular form $\Delta (q)^{n} \haupt_{n,m}$ satisfies the third-degree MLDE:
    \begin{multline}
        \mathcal{D}^{3} ( \Delta (q)^{n} \haupt_{n,m}) = \left( \frac{3}{4}t_{n,m}^{2} + \frac{1}{4}t_{n,m} + \frac{1}{18} \right) \mathcal{D}(\Delta (q)^{n} \haupt_{n,m}) E_{4}(q) \\ - \frac{1}{4} \left( t_{n,m}^{3} + t_{n,m}^{2} \right) \Delta (q)^{n}\haupt_{n,m} E_{6}(q),
    \end{multline}
    where $t_{n,m} = 2^{m}-n$.
\end{Prop}
\begin{proof}
    As in the proof of Lemma \ref{Lemma: hdiffeq}, it suffices to show that the $\haupt_{n,m}$ satisfy the MLDE in the claim with the modular derivative acting with weight $4a \cdot 2^{m}$. Let us temporarily set
   \begin{align*}
       X = 5E_{4}(q) - 20 E_{2}^{\star}(q)^{2}
   \end{align*}
   so that eq. (\ref{hmsbetter}) becomes
   \begin{align*}
        \haupt_{n,m} &= \haupt^{n}  X^{3 \cdot 2^{m}} + 2^{1 - 12n} (240)^{3 \cdot 2^{m}} \Delta (q)^{2^{m+1}-n} \utwo{ \left( \Delta (q)^{n-2^{m}} \right) } \\
        &:= T_{1} + 2^{1-12n} (240)^{3 \cdot 2^{m}}  T_{2}.
   \end{align*}
   We show that $T_{1}$ and $T_{2}$ satisfy the MLDE in the statement of the proposition separately so that by linearity the result holds. Consider first $T_{1}$. It is easily verified that $\mathcal{D}(X) = (1/3)E_{2}^{\star}(q)X$, so that
   \begin{align}
       \mathcal{D}(T_{1}) = t_{n,m} \haupt^{n} E_{2}^{\star}(q) X^{3 \cdot 2^{m}} \label{T1d1}
   \end{align}
   and from this
   \begin{align*}
       \mathcal{D}^{2}(T_{1}) = \haupt^{n} X^{3 \cdot 2^{m}} \left( \left( t_{n,m}^{2} + \frac{1}{3} t_{n,m} \right) E_{2}^{\star}(q)^{2} - \frac{1}{6} t_{n,m} E_{4}(q) \right).
   \end{align*}
   Lastly, a lengthy computation shows that
   \begin{align*}
       \mathcal{D}^{3}(T_{1}) = \haupt^{n}X^{3 \cdot 2^{m}} \left( \left( t_{n,m}^{3} + t_{n,m}^{2} \right) E_{2}^{\star}(q)^{3} + \left( \frac{1}{18} t_{n,m} - \frac{1}{2} t_{n,m}^{2} \right) E_{2}^{\star}(q)E_{4}(q) \right).
   \end{align*}
   Now, using the identity $E_{6}(q) = -4 E_{2}^{\star}(q)^{3} + 3 E_{2}^{\star}(q) E_{4}(q)$, we can re-arrange
   \begin{multline*}
       \mathcal{D}^{3}(T_{1}) = \haupt^{n} X^{3 \cdot 2^{m}} \left( \left( \frac{3}{4} t_{n,m}^{3} + \frac{1}{4} t_{n,m}^{2} + \frac{1}{18} t_{n,m} \right) E_{2}^{\star}(q)E_{4}(q) \right. \\ \left. - \frac{1}{4}\left( t_{n,m}^{3} + t_{n,m}^{2} \right) \left( -4 E_{2}^{\star}(q)^{3} + 3 E_{2}^{\star}(q)E_{4}(q) \right) \right)
   \end{multline*}
   and so using (\ref{T1d1}) we get that $T_{1}$ satisfies the MLDE in the statement. Now for $T_{2}$: Since $\mathcal{D}_{1/2}(\eta(q)) = 0$, we have
   \begin{align*}
       \mathcal{D}^{i}(T_{2}) = \Delta (q)^{2^{m+1}-n} \mathcal{D}^{i} \left( \utwo{  \Delta (q)^{-t_{n,m}} } \right)
   \end{align*}
   for $i = 1,2,3$. In particular, using the fact that $\theta (\eta (q)) = (1/24) E_{2}(q) \eta (q)$ as well as Lemma \ref{Lemma: Upoperator} (2),
   \begin{align}
        \mathcal{D} \left. \left( \Delta (q)^{-t_{n,m}} \right\rvert \hspace{1pt} U_{2} \right) = - \frac{1}{2} t_{n,m} \utwo{ \left( \Delta (q)^{- t_{n,m}} E_{2}^{\star}(q) \right) }. \label{T2d1}
   \end{align}
   From here, we get
   \begin{align*}
       \mathcal{D}^{2} \left( \utwo{ \Delta (q)^{-t_{n,m}} } \right) = \utwo{ \left( \Delta (q)^{-t_{n,m}} \left( \left( \frac{ t_{n,m}^{2}}{4} - \frac{ t_{n,m}}{8} \right) E_{2}^{\star}(q)^{2} + \frac{ t_{n,m}}{24} E_{4}(q) \right) \right) },
   \end{align*}
   and 
   \begin{multline*}
       \mathcal{D}^{3} \left( \utwo{ \Delta (q)^{- t_{n,m}} } \right) = \left. \left( \Delta (q)^{- t_{n,m}} \left( \left( - \frac{t_{n,m}^{3}}{8} + \frac{3 t_{n,m}^{2}}{16} - \frac{5 t_{n,m}}{144} \right) E_{2}^{\star}(q)^{3} \right.\right. \right. \\
       + \utwo{ \left. \left. \left( - \frac{t_{n,m}^{2}}{16} + \frac{t_{n,m}}{144} \right) E_{2}^{\star}(q) E_{4}(q) \right) \right) }.
   \end{multline*}
   Making use of the identities
       \begin{align*}
           E_{4}(q^{2}) &= \frac{5}{4} E_{2}^{\star}(q)^{2} - \frac{1}{4} E_{4}(q) \\
           E_{6}(q^{2}) &= \frac{3}{8}E_{2}^{\star}(q)E_{4}(q) - \frac{11}{8}E_{2}^{\star}(q)^{3}
       \end{align*}
   alongside Lemma \ref{Lemma: Upoperator} (2), we can re-arrange
   \begin{multline*}
           \mathcal{D}^{3} \left( \utwo{ \Delta (q)^{-t_{n,m}} } \right) = \left. \left( \left( - \frac{15}{32} t_{n,m}^{3} - \frac{5}{32} t_{n,m}^{2} - \frac{5}{144} t_{n,m} \right) \Delta (q)^{-t_{n,m}}E_{2}^{\star}(q)^{3} \right. \right. \\ 
           \left. \left. + \left( \frac{3}{32} t_{n,m}^{3} + \frac{1}{32} t_{n,m}^{2} + \frac{1}{144} t_{n,m} \right) \Delta (q)^{- t_{n,m}} E_{2}^{\star}(q)E_{4}(q) \right. \right. \\
           \left. \left. - \frac{3}{32} \left( t_{n,m}^{3} + t_{n,m}^{2} \right) \Delta (q)^{- t_{n,m}} E_{2}^{\star}(q)E_{4}(q) \right. \right. \\
           \utwo{ \left. + \frac{11}{32} \left( t_{n,m}^{3} + t_{n,m}^{2} \right) \Delta (q)^{- t_{n,m}} E_{2}^{\star}(q)^{3} \right) }
       \end{multline*}
       so that by (\ref{T2d1})
       \begin{multline*}
           \mathcal{D}^{3} \left( \utwo{ \Delta (q)^{- t_{n,m}} } \right) = \left( \frac{3}{4} t_{n,m}^{2} + \frac{1}{4} t_{n,m} + \frac{1}{18} \right) \mathcal{D} \left( \utwo{ \Delta (q)^{- t_{n,m}} } \right) E_{4}(q) \\
           - \frac{1}{4} \left( t_{n,m}^{3} + t_{n,m}^{2} \right) \left( \utwo{ \Delta (q)^{- t_{n,m}} } \right) E_{6}(q)
       \end{multline*}
       which shows that $T_{2}$ satisfies the MLDE in the proposition.
\end{proof}

\begin{Rem}
The MLDE given in the statement of Lemma \ref{Lemma: hdiffeq} is the $2$-adic limit of the sequence of MLDEs whose $m$-th term is given in Proposition \ref{Prop: hmdiffs}. This can easily be shown by writing eq. (\ref{hdiffeq}) in the basis $\mathbb{C}[E_{2}^{\star},E_{4}]$ of $M(\Gamma_{0}(2))$, using the expression for $\mathcal{D}(\haupt^{n})$ and noting that $\lim_{m \to \infty} t_{n,m}^{r} = (-n)^{r}$.
\end{Rem}

\begin{Rem}
    For $p \in \lbrace 3,5,7,13 \rbrace$, we are not able to find an MLDE which is satisfied by the $\haupt_{n,m}$. When $p=3$ for example, we can express
    \begin{multline*}
        \haupt_{n,m} = \frac{\eta (q^{3})^{12n}}{\eta (q)^{12n}} \cdot (-80)^{3^{m+1}} \left( \left( \frac{\eta (q)^{3}}{\eta (q^{3})} \right)^{4} + 9q \cdot \left( \frac{\eta (q)^{9}}{\eta (q)} \right)^{3} \left( \frac{\eta (q)^{3}}{\eta (q^{3})} \right)^{4} \right)^{3^{m+1}} \\ + 3^{1-6n} \cdot (-240)^{3^{m+1}} \left( \frac{\eta (q)^{12n}}{\eta (q^{3})^{12n}} \right) \uthree{ \left( \eta (q^{3})^{8} + 9 \eta (q^{3})^{8} \left( \frac{\eta (q^{9})}{\eta (q)} \right)^{3} \right)^{3^{m+1}} }
    \end{multline*}
    and there are similar expressions for the other odd primes, which involve eta-quotients.
\end{Rem}

\subsection{Explicit Formulations}

Given our description of each $\Delta (q)^{n} \haupt_{n,m}$ as a solution of the third-degree MLDE given in Proposition \ref{Prop: hmdiffs}, we follow Section 4.3 of \cite{FM16} in order to describe these in terms of hypergeometric series: In the notation of this paper, we can find that
\begin{align*}
    r_{1} = 2n, \hspace{20pt} r_{2} = \frac{3}{2} \cdot 2^{m} + \frac{n}{2}, \hspace{20pt} r_{3} = \frac{3}{2} \cdot 2^{m} + \frac{n}{2} + \frac{1}{2}
\end{align*}
so that precisely one of the three solutions of the MLDE is modular of level one, and given by
\begin{multline*}
    \Delta (q)^{2^{m}+n} K^{-2^{m}+n} \hspace{1pt} _{3}F_{2} \left( - 2^{m}+n,- 2^{m}+n +\tfrac{1}{3}, - 2^{m}+n+\tfrac{2}{3} \right. \\  \left. ; -\tfrac{3}{2}\cdot 2^{m} + \tfrac{3}{2}\cdot n + 1, -\tfrac{3}{2} \cdot 2^{m} + \tfrac{3}{2} \cdot n + \tfrac{1}{2} ; K \right),
\end{multline*}
where $K=1728/j$. Thus $\Delta (q)^{n} \haupt_{n,m}$ is equal to the above expression up to some constant $C_{n,m}$. We assume that $n < \cdot 2^{m}$ throughout, which we can do since we will eventually be taking the limit as $m \to \infty$. The hypergeometric series above is a specialization of the form studied in the main theorem of \cite{KT01}: Setting $k = 12 \cdot 2^{m} - 12n$ and $\lambda = 2$ in their notation and ensuring the necessary conditions on these parameters hold, the above expression then gives
\begin{align}
    \Delta (q)^{n}\haupt_{n,m} = C_{n,m}\Delta (q)^{2n} \left( \sum_{j=0}^{2^{m}-n} \frac{(3n- 3 \cdot 2^{m})_{3j}}{\psi_{j}} E_{4}(q)^{3 \cdot 2^{m} - 3n - 3j} \Delta (q)^{j} \right) \label{prehmns}
\end{align}
where
\begin{align}
   \psi_{j} &= 2^{-12j} (j!) \prod_{l=1}^{j} \left( 8l-12\cdot 2^{m} + 12n \right) \left( 8l-12 \cdot 2^{m} + 12n-4 \right) \nonumber \\
    &= 2^{-8j} (j!) (3n-3\cdot 2^{m}+1)_{2j} \label{psicoeff}
\end{align}
\begin{Prop} \label{Prop: hmnseisenstein}
    Let $n,m \geq 1$. We have
    \begin{align*}
    \haupt_{n,m} = \sum_{j=0}^{2^{m}-n} \psi_{j}^{-1} \left( 3n-3 \cdot 2^{m} \right)_{3j} 15^{3 \cdot 2^{m}} E_{4}(q)^{3 \cdot 2^{m} - 3n - 3j} \Delta (q)^{n+j}.
\end{align*}
\end{Prop}
\begin{proof}
    Divide first both sides of (\ref{prehmns}) by $\Delta (q)^{n}$ to get
    \begin{align}
        \haupt_{n,m} = C_{n,m}\Delta (q)^{n} \left( \sum_{j=0}^{2^{m}-n} \frac{(3n- 3 \cdot 2^{m})_{3j}}{\psi_{j}} E_{4}(q)^{3 \cdot 2^{m} - 3n - 3j} \Delta (q)^{j} \right). \label{hnmseisensteineq1}
    \end{align}
    Note that this shows the $\haupt_{n,m}$ are indeed modular of level one. It remains to show that $C_{n,m} = 15^{3 \cdot 2^{m}}$. Recall from (\ref{hms}) that
    \begin{align}
        \haupt_{n,m} = \haupt^{n}\left( E_{4}(q)- 2^{4} E_{4}(q^{2}) \right)^{3 \cdot 2^{m}} + 2^{1-12n} \left. \left( \haupt^{-n} \left( E_{4}(q^{2})-E_{4}(q) \right)^{3 \cdot 2^{m}} \right) \right\rvert \hspace{2pt} U_{2}. \label{hmnseisensteineq2}
    \end{align}
    We equate the above two displays and extract the coefficient of $q^{n}$. Set $j=0$ in (\ref{hnmseisensteineq1}) to see that this coefficient is $C_{n,m}$, since $\Delta (q)^{n} = q^{n} + O(q^{n+1})$. In (\ref{hmnseisensteineq2}) only the first term provides a coefficient of $q^{n}$ since $\haupt^{n} = q^{n} + O(q^{n+1})$ also. Thus $C_{n,m} = (1-2^{4})^{3 \cdot 2^{m}} = 15^{3 \cdot 2^{m}}$ as desired.
\end{proof}
Expanding $\Delta (q) = 8000G_{4}(q)^{3} - 147G_{6}(q)^{2}$ and using the fact that $G_{4}(q) = 240 E_{4}(q)$, define
\begin{multline}
    \Psi_{n,m,i,j} = \binom{n+j}{i} \binom{3 \cdot 2^{m}-3n-2j}{j} \frac{3 \cdot 2^{m}-3n}{3 \cdot 2^{m}-3n-2j} \\ \cdot (-1)^{i+j} \cdot 2^{12 \cdot 2^{m} - 6n-6i+2j} \cdot 3^{6 \cdot 2^{m}-3n+i-3j} \cdot 5^{6 \cdot 2^{m}-3i} \cdot 7^{2i}  \label{bigpsicoeff}
\end{multline}
Thus from Proposition \ref{Prop: hmnseisenstein},
\begin{align*}
    \haupt_{n,m} &= \sum_{j=0}^{2^{m}-n} \sum_{i=0}^{n+j} \Psi_{n,m,i,j} G_{4}(q)^{3\cdot 2^{m}-3i} G_{6}(q)^{2i} \\
    &= \sum_{i=0}^{n} \sum_{j=0}^{2^{m}-n} \Psi_{n,m,i,j} G_{4}(q)^{3\cdot 2^{m}-3i} G_{6}(q)^{2i} + \sum_{i=n+1}^{2^{m} +1} \sum_{j=i-n}^{2^{m}-n} \Psi_{n,m,i,j} G_{4}(q)^{3 \cdot 2^{m}-3i} G_{6}(q)^{2i}.
\end{align*}
That is,
\begin{align*}
    \haupt_{n,m} = \sum_{i=0}^{2^{m} +1} c_{n,m,i} G_{4}(q)^{3 \cdot 2^{m}-3i} G_{6}(q)^{2i}
\end{align*}
where
\begin{align}
    c_{n,m,i} = \begin{cases} \sum_{j=0}^{2^{m}-n} \Psi_{n,m,i,j} & 0 \leq i \leq n \\ \sum_{j= i-n}^{2^{m}-n} \Psi_{n,m,i,j} & n+1 \leq i \leq a \cdot 2^{m} +1. \end{cases} \label{hmcoeffs}
\end{align}
Note that we can also just write $c_{n,m,i} = \sum_{j=0}^{2^{m}-n} \Psi_{n,m,i,j}$ for all $0 \leq i \leq 2^{m} +1$ due to the first binomial coefficient in (\ref{bigpsicoeff}). \par 

\section{Vertex Operator Algebras} \label{Section: VOAS}

\subsection{Pre-images} \label{Subsection: Pre-images}

Next we must make a choice of pre-images for our sequence terms $\haupt_{n,m}$. Let $r,s \in \mathbb{N}$ and define the square-bracket states
    \begin{align}
    \alpha_{r} = \frac{(-1)^{r}}{2^{r}(2r-1)!!}h[-2]^{2r}\textbf{1}, \hspace{10pt} \text{and} \hspace{10pt} \beta_{s} = \frac{2^{s}}{(2s-1)!!}h[-3]^{2s}\textbf{1} \label{alphabetastates}
    \end{align}
where $n!!$ denotes the product of all odd integers up to $n$ and where we set $(-1)!! = 0$. Though the character map $F_{\Heis}$ has a large kernel, the following states in $\Heis$ are natural to consider for the following reason: 

\begin{Prop}\label{Prop: charactermultiplicative}
    Denote by $\Sigma$ the infinite-dimensional vector subspace of $\Heis$ spanned by monomials of the form $\alpha_{r}\beta_{s}$ for $r,s \in \mathbb{N}$. Then, the normalized character map restricted to $\Sigma$ gives an isomorphism of vector spaces
    \begin{align*}
        &F_{\Heis}: \Sigma \to M \\
        &\sigma \mapsto F_{\Heis}(\sigma,q)
    \end{align*}
\end{Prop}
\begin{proof}
    We show first that for all $r,s \geq \mathbb{N}$,
    \begin{align}
        F_{\Heis}(\alpha_{r} \beta_{s},q) = F_{\Heis}(\alpha_{r},q) \cdot F_{\Heis}(\beta_{s},q) = G_{4}(q)^{r} G_{6}(q)^{s}. \label{charactermultiplicativeeq1}
    \end{align}
    This is done using (\ref{MTeq}). Fix $r \in \mathbb{N}$ and consider $F_{\Heis}(h[-2]^{2r}\textbf{1},q)$. In this case, 
    \begin{align*}
        \Phi = \lbrace\underbrace{ 2,2, \ldots, 2 }_{\text{$2r$ times}}\rbrace
    \end{align*}
    and there are $(2r-1)!!$ ways of partitioning $\Phi$ into $r$ parts $\lbrace 2,2 \rbrace$. So,
    \begin{align*}
        F_{\Heis}(h[-2]^{2r}\textbf{1},q) = (2r-1)!! \left( \frac{2(-1)^{3}}{(2-1)!(2-1)!} G_{4}(q) \right)^{r} = (-1)^{r} \cdot 2^{r} (2r-1)!! G_{4}(q)^{r}
    \end{align*}
    which shows that $F_{\Heis}(\alpha_{r},q) = G_{4}(q)^{r}$. The proof of $F_{\Heis}(\beta_{s},q) = G_{6}(q)^{s}$ for $s \in \mathbb{N}$ is similar. For $r,s \in \mathbb{N}$, consider now $F_{\Heis}(h[-2]^{2r}h[-3]^{2s}\textbf{1},q)$. We now have that
    \begin{align*}
        \Phi = \lbrace \underbrace{2,2, \ldots ,2}_{\text{$2r$ times}}, \underbrace{3,3, \ldots ,3}_{\text{$2s$ times}} \rbrace.
    \end{align*}
    Notice that any partition of $\Phi$ above that involves a part of the form $\lbrace 2,3 \rbrace$ produces a factor of $G_{5}(q)=0$. Hence the only partitions yielding non-zero terms are those where the $2$'s and $3$'s are paired among themselves, and there are $(2r-1)!!(2s-1)!!$ such partitions. Thus
    \begin{align*}
        F_{\Heis}(h[-2]^{2r}h[-3]^{2s}\textbf{1},q) &= (2r-1)!!(2s-1)!! \left( \frac{2(-1)^{3}}{(2-1)!(2-1)!} G_{4}(q) \right)^{r} \left( \frac{2(-1)^{4}}{(3-1)!(3-1)!} G_{6}(q) \right)^{s} \\
        &= \frac{2^{r}(-1)^{r}(2r-1)!!(2s-1)!!}{2^{s}} G_{4}(q)^{r}G_{6}(q)^{s}
    \end{align*}
    which proves (\ref{charactermultiplicativeeq1}). From this, it is clear that $F_{\Heis}$ restricted to $\Sigma$ is injective. Let $f \in M_{k}$ for some $k \geq 4$ even. So we can write
    \begin{align*}
        f = \sum_{i=0}^{\lfloor k/12 \rfloor \pm 1 } c_{i} G_{4}(q)^{k/4-3i} G_{6}(q)^{2i}
    \end{align*}
    for some $c_{i} \in \mathbb{C}$. By letting
    \begin{align*}
        v = \sum_{i=0}^{\lfloor k/12 \rfloor \pm 1 } c_{i} \alpha_{k/4-3i} \beta_{2i} \in \Sigma,
    \end{align*}
    from the linearity of the character map we see that $F_{\Heis}(v,q) = f$, establishing that $F_{\Heis}$ gives the required isomorphism.
\end{proof}

Proposition \ref{Prop: charactermultiplicative} alongside the linearity of the character map show that for all $n,m \geq 1$ with $a=3$, the states
\begin{align}
    v_{n,m} := \sum_{i=0}^{2^{m} +1} c_{n,m,i} \alpha_{3\cdot 2^{m} - 3i} \beta_{2i}
\end{align}
where the coefficients $c_{n,m,i} \in \mathbb{Z}$ are given in (\ref{hmcoeffs}), satisfy $F_{\Heis}(v_{n,m},q) = \haupt_{n,m}$. \par 

\subsection{Convergence of Pre-Images} \label{Subsection: Convergence of Pre-Images}

The purpose of this subsection is to develop results necessary for the proof of Proposition \ref{Prop: overconvergent}, namely that for all $n,m \geq 1$, the pre-images $v_{n,m}$ satisfy $\nu_{2}(v_{n,m}) \geq -6n$. These results involve a close examination of the $2$-adic properties of the coefficients $c_{n,m,i}$. In contrast to \cite{BF24,FM23a,FM23}, we avoid obtaining explicit expressions for the $v_{n,m}$ in the round-bracket formalism. In fact, in Lemma \ref{Lemma: alphabetaval}, we prove a useful tool for assessing the valuation of the pre-images $\alpha_{r}\beta_{s}$ (cf. (\ref{alphabetastates})) for any $r,s \geq 2$. This will be used repeatedly towards the proof of Proposition \ref{Prop: overconvergent}.

\begin{Lemma} \label{Lemma: hmcoeffsvals}
    Let $n,m \geq 1$. When $0 \leq i \leq n$, we have
    \begin{align*}
        \nu_{2}(c_{n,m,i}) \geq 12 \cdot 2^{m} - 6n - 6i,
    \end{align*}
    and when $n+1 \leq i \leq 2^{m}-n$ we have
    \begin{align*}
        \nu_{2}(c_{n,m,i}) \geq 12 \cdot 2^{m}-8n-4i.
    \end{align*}
\end{Lemma}
\begin{proof}
From (\ref{bigpsicoeff}),
    \begin{multline*}
        \nu_{2} \left( \sum_{j=0}^{2^{m}-n} \Psi_{n,m,i,j} \right) = \nu_{2} \left( \sum_{j=0}^{2^{m}-n} \binom{n+j}{i} \binom{3 \cdot 2^{m}-3n-2j}{j} \frac{3\cdot 2^{m}-3n}{3 \cdot 2^{m}-3n-2j} \cdot (-1)^{j} \cdot 2^{2j} \cdot 3^{-3j} \right) \\ + 12 \cdot 2^{m}-6n-6i
    \end{multline*}
    Suppose $0 \leq i \leq n$. Since each term in the sum above has non-negative valuation, we can bound
    \begin{align*}
        \nu_{2}(c_{n,m,i}) \geq 12 \cdot 2^{m} - 6n - 6i.
    \end{align*}
    Likewise when $n+1 \leq i \leq 2^{m}+1$, we have $i-n \leq j \leq 2^{m}-n$. So it is not hard to see that we can write
    \begin{align*}
        \nu_{2}(c_{n,m,i}) \geq 12 \cdot 2^{m}-8n-4i.
    \end{align*}
\end{proof}

\begin{Lemma} \label{Lemma: alphabetaval}
    Let $r,s$ be integers with $r,s \geq 2$. Then
    \begin{align*}
        \nu_{2} \left( \alpha_{r} \beta_{s} \right) = -4r-3s.
    \end{align*}
\end{Lemma}
\begin{proof}
    We write (\ref{squareh2exp}) and (\ref{squareh3exp}) as the following commuting operators:
    \begin{align*}
        X &= h(-2) + h(-1) - \frac{1}{120}\partial_{h(-2)} + \frac{1}{120} \partial_{h(-3)} \\
        Y &= h(-3) + \frac{3}{2}h(-2) + \frac{1}{2}h(-1) + \frac{1}{240} \partial_{h(-1)} - \frac{1}{240}\partial_{h(-2)} + \frac{1}{315} \partial_{h(-3)}.
    \end{align*}
    The fact that these commute is direct from the property that $h[-3]$ and $h[-2]$ commute. We see that for $r,s \in \mathbb{N}$, $\nu_{2}(X^{r}Y^{s}\textbf{1}) = -3r-4s$ since $\nu_{2}(120)=3$ and $\nu_{2}(240)=4$. So from the definitions of $\alpha_{r}$ and $\beta_{s}$ in (\ref{alphabetastates}), we have $\nu_{2}(\alpha_{r}\beta_{s}) = -4r-3s$ as desired.
\end{proof}
\noindent We now establish that the sequence $\lbrace v_{n,m} \rbrace_{m \geq 1}$ for fixed $n \geq 1$ is Cauchy. Write
\begin{align}
    \nu_{2}&(v_{n,m+1} - v_{n,m}) \nonumber \\
    &= \nu_{2}\left( \sum_{i=0}^{2^{m+1}} c_{n,m+1,i}\alpha_{3 \cdot 2^{m+1}-3i}\beta_{2i} - \sum_{j=0}^{2^{m}} c_{n,m,j} \alpha_{3\cdot 2^{m}-3j}\beta_{2j} \right) \nonumber \\
    &\geq \min \left\lbrace \nu_{2}\left( \sum_{i=0}^{2^{m}} c_{n,m+1,i}\alpha_{3\cdot 2^{m+1}-3i} \beta_{2i} - c_{n,m,i}\alpha_{3\cdot 2^{m}-3i}\beta_{2i} \right), \nu_{2}\left(\sum_{j=2^{m}+1}^{2^{m+1}} c_{n,m+1,j}\alpha_{3\cdot 2^{m+1}-3j}\beta_{2j} \right) \right\rbrace. \label{cauchytwosums}
\end{align}
Note that $\nu_{2}(\alpha_{3\cdot 2^{m+1}-3j}\beta_{2j}) = -12\cdot 2^{m+1}+6j$ from Lemma \ref{Lemma: alphabetaval}. Since we have assumed $n < 2^{m}$, from Lemma \ref{Lemma: hmcoeffsvals}, whenever $2^{m}+1 \leq j \leq 2^{m+1}$ we get
\begin{align*}
    \nu_{2}(c_{n,m+1,j}\alpha_{3\cdot 2^{m+1}-3j}\beta_{2j}) \geq 2 j -8n.
\end{align*}
Thus the second term in (\ref{cauchytwosums}) tends to infinity as $m \to \infty$. We can write
\begin{align*}
    \alpha_{3\cdot 2^{m+1}-3i} = \frac{(6\cdot 2^{m}-6i-1)!!(6\cdot 2^{m}-1)!!}{(6\cdot 2^{m+1}-6i-1)!!} \alpha_{3 \cdot 2^{m}} \alpha_{3 \cdot 2^{m}-3i}
\end{align*}
where the coefficient above has $2$-adic valuation of zero. With this, the first term in (\ref{cauchytwosums}) becomes
\begin{align*}
    \nu_{2}\left( \sum_{i=0}^{2^{m}} \alpha_{3 \cdot 2^{m}-3i}\beta_{2i} \left( c_{n,m+1,i} \alpha_{3 \cdot 2^{m}} - c_{n,m,i} \right) \right).
\end{align*}
Using Lemma \ref{Lemma: alphabetaval} once again, it remains to prove that 
\begin{align*}
    \nu_{2}\left( c_{n,m+1,i} \alpha_{3 \cdot 2^{m}} - c_{n,m,i} \right) - 12\cdot 2^{m}+6i \to \infty
\end{align*}
for all $0 \leq i \leq 2^{m}$ as $m \to \infty$. We require a small lemma first: \par 
\begin{Lemma} \label{Lemma: alphaformula}
    Let $r \geq 1$. Then
    \begin{align*}
        \alpha_{r} = \frac{(-1)^{r}}{2^{r}(2r-1)!!} \sum_{l=0}^{r} \binom{2r}{2l} \frac{(2l)!}{(-240)^{l}l!}\left( h(-2)+h(-1) \right)^{2(r-l)}\textbf{1}.
    \end{align*}
\end{Lemma}
\begin{proof}
    We obtain a formula for $h[-2]^{2r}\textbf{1}$. It was shown in Section 3.2 of \cite{BF24} that the Rodrigue's formula for the generalized Hermite polynomials
    \begin{align}
        \text{He}_{r}^{\alpha}(x) = (x- \gamma \partial_{x})^{r} \cdot 1 = \sum_{l \geq 0} \binom{r}{2l}\frac{(-\gamma)^{l}(2l)!}{2^{l}(l)!}x^{r-2l} \label{hermiteformula}
    \end{align}
    with $r \geq 0$ and $\gamma > 0$ can be used to prove that
    \begin{align*}
        h[-1]^{r}\textbf{1} = \sum_{l \geq 0} \binom{r}{2l} \frac{(2l)!}{l!(-24)^{l}} h(-1)^{r-2l}\textbf{1}
    \end{align*}
    by setting $x = h(-1)$ and $\gamma = 1/12$. By setting
    \begin{align*}
        W = h(-2)+h(-1)-\frac{1}{120}\partial_{h(-2)}
    \end{align*}
    (cf. Lemma \ref{Lemma: alphabetaval}), we see that $h[-2]^{2r}\textbf{1} = W^{2r}\textbf{1}$. So by letting $\gamma = 1/120$ and $x=h(-2)+h(-1)$ in (\ref{hermiteformula}) and using the definition of $\alpha_{2r}$ in (\ref{alphabetastates}) we get the desired expression.
\end{proof}
From above, as $0 \leq l \leq r$, the valuation of the coefficient of $\left( h(-2)+h(-1) \right)^{2(r-l)}$ is decreasing, and reaches a minimum at the constant term when $l=r$ (cf. Lemma \ref{Lemma: alphabetaval}). In particular, this shows that
\begin{multline*}
    \nu_{2}\left( c_{n,m+1,i} \alpha_{3 \cdot 2^{m}} - c_{n,m,i} \right) \\ = \min \left\lbrace \nu_{2}\left( c_{n,m+1,i} \text{Coeff}_{h(-1)^{2}\textbf{1}} \alpha_{3 \cdot 2^{m}} \right),   \nu_{2} \left( c_{n,m+1,i} \left( \text{Coeff}_{\textbf{1}} \alpha_{3 \cdot 2^{m}} \right) - c_{n,m,i} \right) \right\rbrace.
\end{multline*}
\begin{Prop} 
    Fix $n \geq 1$. Then for all $0 \leq i \leq 2^{m}+1$, and as $m \to \infty$ we have
    \begin{align*}
        \nu_{2}\left( c_{n,m+1,i} \alpha_{3 \cdot 2^{m}} - c_{n,m,i} \right) -12 \cdot 2^{m}+6i \to \infty 
    \end{align*}
\end{Prop}
\begin{proof}
    From Lemma \ref{Lemma: alphaformula},
    \begin{align*}
        \text{Coeff}_{\textbf{1}} \alpha_{3 \cdot 2^{m}} &= \frac{(6 \cdot 2^{m})!}{(480)^{3 \cdot 2^{m}} (3 \cdot 2^{m})! (6 \cdot 2^{m}-1)!!} \\
        &= 2^{-12\cdot 2^{m}} \cdot 3^{-3 \cdot 2^{m}} \cdot 5^{-3 \cdot 2^{m}}, \\
        \intertext{and}
        \text{Coeff}_{h(-1)^{2}\textbf{1}} \alpha_{3 \cdot 2^{m}} &= -45\cdot 2^{m+3} \left( \frac{(6 \cdot 2^{m})!}{(480)^{3 \cdot 2^{m}} (3 \cdot 2^{m})! (6 \cdot 2^{m}-1)!!} \right) \\
        &= -2^{-12 \cdot 2^{m} + m + 3} \cdot 3^{-3 \cdot 2^{m}+2} \cdot 5^{-3 \cdot 2^{m} +1}.
    \end{align*}
    Suppose that $0 \leq i \leq n$. It is not difficult to compute using Lemma \ref{Lemma: hmcoeffsvals} that
    \begin{align}
        \nu_{2}\left( c_{n,m+1,i} \text{Coeff}_{h(-1)^{2}\textbf{1}} \alpha_{3 \cdot 2^{m}} \right) \geq 12 \cdot 2^{m} - 6n - 6i + m + 3. \label{propineq1}
    \end{align}
    Likewise, when $n+1 \leq i \leq 2^{m}+1$, we get
     \begin{align}
        \nu_{2}\left( c_{n,m+1,i} \text{Coeff}_{h(-1)^{2}\textbf{1}} \alpha_{3 \cdot 2^{m}} \right) \geq 12 \cdot 2^{m} - 8n - 4i + m + 3. \label{propineq2}
    \end{align}
    Now for the constant terms:
    \begin{align*}
         c_{n,m+1,i} \text{Coeff}_{\textbf{1}} \alpha_{3 \cdot 2^{m}} - c_{n,m,i} &= 2^{-12\cdot 2^{m}} \cdot 3^{-3 \cdot 2^{m}} \cdot 5^{-3 \cdot 2^{m}} \sum_{j=0}^{2^{m+1}-n} \Psi_{n,m+1,i,j} - \sum_{j=0}^{2^{m}-n}\Psi_{n,m,i,j}.
    \end{align*}
   which have valuation equal to
   \begin{multline*}
       \nu_{2} \left( 3^{3 \cdot 2^{m}} \cdot 5^{3 \cdot 2^{m}} \sum_{j=0}^{2^{m+1}-n} \binom{n+j}{i} \binom{6 \cdot 2^{m}-3n-2j}{j} \frac{6 \cdot 2^{m}-3n}{6 \cdot 2^{m}-3n-2j} (-1)^{j} \cdot 2^{2j} \cdot 3^{-3j} \right. \\
       \left. - \sum_{j=0}^{2^{m}-n} \binom{n+j}{i} \binom{3 \cdot 2^{m}-3n-2j}{j} \frac{3 \cdot 2^{m}-3n}{3 \cdot 2^{m}-3n-2j} (-1)^{j} \cdot 2^{2j} \cdot 3^{-3j} \right) + 12 \cdot 2^{m}-6n-6i.
   \end{multline*}
    Consider the valuation of the difference of sums above. When $0 \leq i \leq n$, the minimum is reached when $j=0$ so that
    \begin{align}
        \nu_{2} \left( c_{n,m+1,i} \text{Coeff}_{\textbf{1}} \alpha_{3 \cdot 2^{m}} - c_{n,m,i}  \right) &= 12 \cdot 2^{m}-6n-6i + \nu_{2} \left( \binom{n}{i} \left( (1-2^{4})^{3 \cdot 2^{m}} -1 \right) \right) \nonumber \\
        &= 12 \cdot 2^{m}-6n-6i + m+4 + \nu_{2} \left( \binom{n}{i} \right). \label{propineq3}
    \end{align}
    Likewise when $n+1 \leq i \leq 2^{m}+1$, the minimum valuation is attained when $j=i-n$:
    \begin{multline*}
        \nu_{2} \left( c_{n,m+1,i} \text{Coeff}_{\textbf{1}} \alpha_{3 \cdot 2^{m}} - c_{n,m,i}  \right) = 12 \cdot 2^{m}-8n-4i \\ + \nu_{2} \left( \binom{6 \cdot 2^{m} - n -2i}{i-n} \frac{6 \cdot 2^{m}-3n}{6 \cdot 2^{m}-n-2i} - \binom{3 \cdot 2^{m}-n-2i}{i-n} \frac{3 \cdot 2^{m}-3n}{3 \cdot 2^{m}-n-2i} \right)
    \end{multline*}
    so that
    \begin{align}
        \nu_{2} \left( c_{n,m+1,i} \text{Coeff}_{\textbf{1}} \alpha_{3 \cdot 2^{m}} - c_{n,m,i}  \right) &\geq 12 \cdot 2^{m}-8n-4i + m+2-2i+2n \nonumber \\
        &= 12 \cdot 2^{m}-6n-6i+m+2. \label{propineq4}
    \end{align}
  Looking at all four cases (\ref{propineq1}), (\ref{propineq2}), (\ref{propineq3}) and (\ref{propineq4}), we see that the claim in the proposition holds.
\end{proof}

\subsection{Proof of Theorems \ref{Th: TheoremImage} and \ref{Th: TheoremImage2}} \label{Subsection: OverconvergentImage}

We showed in the previous subsection that for all $n \geq 1$, the limits $v_{n} = \lim_{m \to \infty} v_{n,m}$ are elements of the $2$-adic Heisenberg VOA $\pHeis$. Thus the space $\mathbb{Q}( \haupt )\subset \twomodforms$ is in the image of the character map on $\pHeis$ which proves Theorem \ref{Th: TheoremImage2}.

\begin{Prop} \label{Prop: overconvergent}
    Let $n \geq 1$. Then $\nu_{2}\left( v_{n,m} \right) \geq -6n$.
\end{Prop}
\begin{proof}
    We use Lemmas \ref{Lemma: hmcoeffsvals} and \ref{Lemma: alphabetaval}. When $0 \leq i \leq n$,
    \begin{align*}
        \nu_{2}\left( c_{n,m,i} \alpha_{3 \cdot 2^{m}-3i}\beta_{2i} \right) \geq -6n
    \end{align*}
    and when $n+1 \leq i \leq 2^{m}+1$,
    \begin{align*}
        \nu_{2}\left( c_{n,m,i} \alpha_{3 \cdot 2^{m}-3i}\beta_{2i} \right) \geq -8n+2i. 
    \end{align*}
    Noting that $-8n+2i$ attains a minimum value at $-6n+2$ when $i=n+1$, it is easy to see that
    \begin{align*}
        \nu_{2}\left( v_{n,m} \right) \geq \min_{0 \leq i \leq 2^{m}} \left\lbrace \nu_{2} \left( c_{n,m,i}\alpha_{3 \cdot 2^{m}-3i}\beta_{2i} \right) \right\rbrace = -6n.
    \end{align*}
\end{proof}

Hence the statement of Theorem \ref{Th: TheoremImage} holds immediately from Proposition \ref{Prop: overconvergent} and the discussion at the end of Subsection \ref{Subsection: padicforms}.

\bibliographystyle{amsalpha}
\bibliography{references}

\end{document}